\documentclass[10pt, reqno]{amsart}
\usepackage{amscd} 
\usepackage{amsfonts} 
\usepackage{amssymb} 
\usepackage{latexsym} 
\usepackage[arrow, matrix, curve]{xy}

\newcommand{\ncm}{\newcommand}

\newtheorem{theorem}{Theorem}[section]
\newtheorem{prop}[theorem]{Proposition}

\newtheorem{lemma}[theorem]{Lemma}
\newtheorem{cor}[theorem]{Corollary}
\newtheorem{lem&def}[theorem]{Lemma \& Definition}
\newtheorem{definition}[theorem]{Definition}

\def\Z{\mathbb{Z}\,} 
 
\def\N{\mathbb{N}\,}

\ncm{\Ann}{\mbox{\rm Ann}\,}
\ncm{\End}{\mbox{\rm End}\,}

\def\Hom{\mbox{\rm Hom}\,}

\def\|{\, | \,}

\def\id{\mbox{\rm id}}

\def\into{\hookrightarrow}

\ncm{\rarr}[1]{\stackrel{#1}{\longrightarrow}}
\ncm{\larr}[1]{\stackrel{#1}{\longleftarrow}}

\def\eps{\varepsilon}

\def\-2{_{(-2)}}
\def\-1{_{(-1)}}
\def\0{_{(0)}}
\def\1{_{(1)}}
\def\2{_{(2)}}
\def\3{_{(3)}}

\def\du1{\hat 1}

\begin{document}
\title[Separable equivalence of rings and symmetric algebras]{Separable equivalence of rings \\ and symmetric algebras\footnote{Addendum to \textit{Hokkaido Mathematical J.} \textbf{24} (1995), 527-549}}
\author[Lars Kadison]{Lars Kadison} 
\address{Department of Mathematics \\ University of Pennsylvania \\ 
David Rittenhouse Laboratory, 209 S.\ 33rd St.\\ Philadelphia, PA 19109} 
\email{lkadison@math.upenn.edu } 
\date{} 

\begin{abstract}
We continue a study of separable equivalence from \cite{LK1995, K93}.  We prove that symmetric separable equivalent rings $A$ and $B$ are linked by a Frobenius bimodule 
${}_AP_B$ such that $A$ is $P$-separable over $B$. Separably equivalent rings are linked by a biseparable bimodule $P$.   In addition, the ring monomorphism $A \into \End P_B$ is split, separable Frobenius.   It is observed that left and right finite projective bimodules over symmetric algebras are  Frobenius bimodules; twisted by the Nakayama automorphisms if
over Frobenius algebras.  
\end{abstract} 
\maketitle

\section{Introduction}
Separable equivalence is a weakening of Morita equivalence which symmetrizes the notion of separable,
split ring extensions that are left and right finitely generated projective \cite{K91, K93, LK1995}.   
For finite-dimensional algebras, separable and symmetric separable equivalence are formally defined in \cite{L,P}, with interesting applications of separable equivalence in \cite{L, BE, P} and further publications in modular and ordinary representation theory.  For example, \cite{L} proves that certain Hecke algebras have finitely generated cohomology algebras, \cite{BE} an upper bound on the representation dimension of certain Hecke algebras of types A, B and D, 
and \cite{P} shows that complexity and representation type are preserved by separable equivalence, while the algebras $k[X]/ (X^n)$ are separably inequivalent for certain different natural numbers $n$ and algebraically closed field $k$ \cite[Theorem 8]{P} (from which it follows that the half quantum groups, Taft Hopf algebras $H_n$ are separably inequivalent at the same integers, by an application of   \cite[p.\ 71]{NEFE}). 

The same concept for rings is defined about twenty years before in \cite[Def.\ 6.1]{K91, LK1995} and its tweek in \cite[Prop.\ 6.3]{K91, LK1995}, which seems not to be known to the authors of \cite{L, BE, P}.   One of the purposes of this addendum is to provide clarification and dispel any obscurity, synchronize terminologies, and also  provide the details of the structure referred to in the paragraph before \cite[Prop.\ 6.1]{LK1995}
as well as  \cite{K91}, which is a suitably weakened Morita structure theory.  

We then characterise symmetric separably equivalent rings $A$ and $B$ as being linked by
a Frobenius bimodule ${}_AP_B$, its dual ${}_BQ_A$, such that $A$ is $P$-separable over $B$ and $B$ is $Q$-separable over $A$.  In the terminology of \cite{CK} the bimodule $P$ is
a biseparable Frobenius bimodule; unlike a Morita context bimodule, it is not in general faithfully balanced, nor are the two split bimodule epis $\nu: P \otimes_B Q \rightarrow A$ and $\mu: Q \otimes_A P \rightarrow B$ in general associative.  We show that $P \otimes_B Q \cong \End P_B$ as rings in terms of
a $\mu$-multiplication that generalises and symmetrises E-multiplication \cite{LK1995, K93, NEFE}. Symmetrically, $Q \otimes_A P$ with the $\nu$-multiplication is isomorphic as rings to $\End Q_A$.  We
note that left and right finite projective bimodules linking Frobenius algebras are automatically twisted
Frobenius bimodules, receiving twists from the Nakayama automorphisms. This  simplifies significantly  the proof that separable equivalence and symmetric separable equivalence coincide as notions for symmetric algebras.

\section{Preliminaries on separably divides} 

It is useful to divide the notion of separable equivalence into two equal parts in order to halve the mathematical exposition. 
At the same time, we establish some fixed notation used throughout the paper. 
Suppose $A$ and $B$ are rings.  They are said to be \textit{linked} if there is a bimodule
${}_AP_B$ or a bimodule ${}_BQ_A$.  A bimodule ${}_AP_B$ is said to be left and right finite projective if
 the restricted  modules $P_B$ and ${}_AP$ are  finitely generated and projective.  
For example,
the natural bimodule ${}_AA_A$ is left and right finite projective, but  projective if and only if $A$ is a separable $\Z$-algebra.  
\begin{definition}
The ring $A$ \textbf{separably divides} $B$ if they are linked by left and right finite projective bimodules ${}_AP_B$
and ${}_BQ_A$, and there is 
a split epimorphism $\nu: P \otimes_B Q \rightarrow A$ of $A$-$A$-bimodules \cite{BE}.  
The ring $A$ \textbf{symmetric separably divides} $B$ if $A$ separably divides $B$ and additionally
the functors $\mathcal{P} = {}_AP \otimes_B -$ and $\mathcal{Q} = {}_B  Q \otimes_A -$
are an adjoint pair $(\mathcal{P}, \mathcal{Q})$ between the categories of all modules, $A$-Mod
and $B$-Mod, with $\nu$ the counit of adjunction.  
\end{definition}

The data $(P,Q,\nu)$ will be called a \textit{context} as in Morita theory. The \textit{center} of a bimodule ${}_AW_A$ is the abelian group $W^A$ of elements $w \in W$ such that
$aw = wa$ for all $a \in A$. Note that $\nu: P \otimes Q \rightarrow A$ being a split
epi of $A$-$A$-bimodules is equivalent to there being 
a (\textit{separability}) element $\sum_i z_i \otimes w_i$ in $(P \otimes_B Q)^A$ such that $\sum_i \nu(z_i \otimes w_i) = 1_A$.

\begin{lemma}
The context modules $Q_A$ and ${}_AP$ are generators. The natural mappings $\lambda: A \rightarrow \End P_B$ and $\rho: A \rightarrow \End {}_BQ$ are injective.  
\end{lemma}
\begin{proof}
Note that the mapping $P \rightarrow \Hom (Q_A,A_A)$ given by $p \mapsto \nu(p \otimes_B -)$ is an $A$-$B$-bimodule homomorphism.  Since $\nu$ is surjective,  it follows that there are
$z_i \in P$, $w_i \in Q$ such that $\sum_i \nu(z_i \otimes w_i) = 1_A$, i.e., the trace ideal of $Q_A$ is all of $A$. Thus, $Q_A$ is a generator.  One shows similarly that  ${}_AP$ is a generator.

The natural mapping is of course given by $\lambda_a(p) = ap$.  If $\lambda_a = 0$ in $\End P_B$,
then $a$ belongs to the annihilator of the $A$-module $P$, a generator, so ${}_AA \| n \cdot {}_AP$ for some
$n \in \N$, and
$a = 0$ when applied to $1_A$. Thus, $\lambda$ is injective; similarly, $\rho$ is injective.  
\end{proof}
The mappings $\lambda$ and $\rho$ are isomorphisms if the context bimodules are faithfully balanced; for example, Morita context bimodules \cite{AF, Lam}. 
Recall the notion of a separable functor \cite{R}, 
which is a functor $F: \mathcal{C} \rightarrow \mathcal{D}$  with right adjoint $G: \mathcal{D} \rightarrow \mathcal{C}$ such that unit of adjunction
$\eta: 1_{\mathcal{C}}  \stackrel{.}{\rightarrow} GF$ splits \cite[Theorem 1.2]{IS}.  Equivalently,
a separable functor $F$ reflects split exact sequences. 
\begin{theorem}
Ring $A$ separably divides $B$ if and only if the functor $\mathcal{Q} = {}_BQ \otimes_A - $ from $A$-Mod into $B$-Mod, where the bimodule ${}_BQ_A$ is left and right finite projective,  is a separable functor.
\end{theorem}
\begin{proof}
($\Leftarrow$) Note that the functor $F = \mathcal{Q} = Q \otimes_A -$ has right adjoint functor $G = \Hom ({}_BQ, {}_B-)$ from $B$-Mod into $A$-Mod \cite{AF}.  Since $Q$ is left and right finite projective, $G$ is naturally isomorphic to ${}_AP \otimes_B - $ where $P := \Hom ({}_BQ, {}_BB)$.  The unit of adjunction is the natural transformation $\eta_M: M \rightarrow \Hom ({}_BQ, {}_BQ \otimes_A M)$ that for each $A$-module $M$ (and $m \in M$) sends
\begin{equation}
\label{eq: uofa}
m \longmapsto (q \mapsto q \otimes_A m) 
\end{equation}
which is obviously natural with respect to arrows ${}_AM \rightarrow {}_AN$.  Assume that the unit of adjunction
splits, i.e. has a  section natural transformation  $\nu_M: \Hom ({}_BQ, {}_BQ \otimes_A M) \rightarrow M$, natural with respect to arrows ${}_AM \rightarrow {}_AN$ and satisfying $\nu_M \circ \eta_M = \id_M$ for every ${}_AM$. Then $$\nu := \nu_A: \End {}_BQ \cong
\Hom ({}_BQ, {}_BB) \otimes_B Q \cong P \otimes Q  \longrightarrow A$$
is split by the mapping $\eta_A(a) = \rho_a$ where $\rho_a(q) = qa$ for each $q \in Q, a \in A$.  By naturality and a usual argument, $\nu$ is an $A$-$A$-bimodule split epi.  

($\Rightarrow$) Given a split bimodule epi $\nu': {}_AP \otimes_B Q_A \rightarrow {}_AA_A$, denote
values of $\nu'$ by $\nu'(p \otimes q) = p \cdot q \in A$. Then there is a separability element
$\sum_i z_i \otimes w_i$ where $z_i \in P, w_i \in Q$.  Given the unit of adjunction $\eta_M$ in Eq.~(\ref{eq: uofa})
for the functor $\mathcal{Q}$ with right adjoint $\Hom ({}_BQ, {}_B-)$, we define a natural splitting
or section $\nu_M: \Hom ({}_BQ, {}_BQ \otimes_A M) \rightarrow M$ for each module ${}_AM$, by 
$f \mapsto \sum_i z_i \cdot f(w_i) \in M$.  This satisfies $\nu_M \circ \eta_M = \id_M$ up to natural isomorphism $A \otimes_A M \cong M$, since $\sum_i z_i \cdot w_i = 1_A$. It follows that
$\mathcal{Q}$ is a separable functor. 
\end{proof}
  The next proposition explains more fully the ``elements of adjunction'' in \cite[Def.\ 6.1]{LK1995}, which are useful for computations.  

\begin{prop}
\label{prop-adjelts}
Given bimodules ${}_AP_B,{}_BQ_A$ and bimodule homomorphism $ \nu: {}_AP\otimes_B Q_A \rightarrow {}_AA_A$, the functor $\mathcal{P} =  P \otimes_B -$: $B$-Mod $\rightarrow A$-Mod is a left adjoint of $\mathcal{Q} = Q \otimes_A - $ with counit of adjunction $\nu$ if and only if there is an element  $\sum_i {q'}_i \otimes {p'}_i \in (Q \otimes_A P)^B$  such that for every $p \in P, q \in Q$ we have 
\begin{eqnarray}
\label{eq1}
\sum_i \nu(p \otimes {q'}_i)  {p'}_i &=&  p  \\
 \label{eq2} 
\sum_i {q'}_i \nu({p'}_i \otimes q) &=&  q
\end{eqnarray}
\end{prop}
\begin{proof}
($\Rightarrow$)  The eqs.~(\ref{eq1}) and~(\ref{eq2}) are the two triangular identities \cite[p.\ 85]{M} characterising adjunction, for the unit of adjunction $\eta_B: b \mapsto \sum_i b{q'}_i \otimes_A {p'}_i$, $\eta: 1 \stackrel{\cdot}{\rightarrow} \mathcal{Q}\mathcal{P}$, and the counit of adjunction $\eps_A = \nu$, $\eps: \mathcal{P}\mathcal{Q} \stackrel{\cdot}{\rightarrow} 1$. The triangular identities are given functorially by $(\eps \mathcal{P}) (\mathcal{P} \eta) = 1_{\mathcal{P}}$ and $(\mathcal{Q} \eps)(\eta \mathcal{Q}) = 1_{\mathcal{Q}}$.  In the usual set-up of adjoint functors, there is a natural  isomorphism
$\Hom ({}_AP \otimes_B M,{}_AN) \cong \Hom ({}_BM, {}_BQ \otimes_A N)$ for each $B$-module $M$
and  $A$-module $N$.  Letting $M = {}_BQ$ and $N = {}_AA$ obtains counit of adjunction $\nu: P\otimes_B Q \rightarrow A$ as the inverse image of
$\id_Q$, which is in $\End ({}_BQ_A)$, so by naturality $\nu$ is a bimodule homomorphism.  Similarly, 
letting $M = {}_BB$ and $N = {}_AP$ obtains the unit of adjunction  $\sum_i {q'}_i \otimes_A {p'}_i \in Q \otimes_A P$, up to usual identifications, as the
image of $\id_P$, and is in the center.  

($\Leftarrow$) This is a healthy pedestrian exercise showing 
\begin{equation}
\label{eq: adjunction}
\Hom ({}_AP \otimes_B M,{}_AN) \cong \Hom ({}_BM, {}_BQ \otimes_A N)
\end{equation}
given eqs.~(\ref{eq1}) and ~(\ref{eq2}). 
 The forward mapping is  $$\Phi_{M,N}: f \longmapsto (m \mapsto
\sum_i {q'}_i \otimes_A f({p'}_i \otimes_B m)) $$
for arbitrary $f: {}_AP \otimes_B M \rightarrow {}_AN$. The 
inverse mapping is 
$$\Psi_{M,N}: g \longmapsto (p \otimes m \mapsto (\nu \otimes \id_N)(p \otimes g(m)) $$
for arbitrary $g: {}_BM \rightarrow {}_BQ \otimes_A N$. 
 Eq.~(\ref{eq1}) implies $\Psi_{M,N} (\Phi_{M,N}(f)) = f$ and eq.~(\ref{eq2}) implies $\Phi_{M,N}( \Psi_{M,N}(g)) = g$.   
\end{proof}
Note that the proof does not make use of the conditions, split epi or epi, on the homomorphism $\nu$ used earlier.    Also note that 
\begin{equation}
\label{eq: centered}
\sum_i {q'}_i \otimes_A {p'}_i \in (Q \otimes_A P)^B
\end{equation}
follows from eqs.~(\ref{eq1}) and ~(\ref{eq2}) 
since $\sum_i b{q'}_i \otimes_A {p'}_i = \sum_{i,j} {q'}_j \nu({p'}_j \otimes_B b{q'}_i) \otimes_A {p'}_i $
$$= \sum_{i,j} {q'}_j \otimes_A \nu({p'}_j b \otimes_B {q'}_i){p'}_i = \sum_j {q'}_j \otimes_A {p'}_j b .$$
The next two lemmas use only the hypotheses of Proposition~\ref{prop-adjelts}. 
\begin{lemma}
\label{lemma-oneforall}
Given adjoint functors $(\mathcal{P},\mathcal{Q})$, the counit of adjunction is a split epi if any one bimodule homomorphism $\phi: P \otimes_B Q \rightarrow A$ is a split epi.
\end{lemma}
\begin{proof}
From Eq.~(\ref{eq: adjunction}) with $M = {}_BQ$ and $N = {}_AA$, we obtain from naturality  
\begin{equation}
\Hom ({}_AP \otimes_B Q_A, {}_AA_A) \cong \End {}_BQ_A
\end{equation}
 as abelian groups via for each $g \in \End {}_BQ_A$, 
$\Psi_{Q,A}(g) = \nu(P \otimes g) = \phi: P \otimes_B Q \rightarrow A$.  If $h$ is a splitting for $\phi$ then $(P \otimes g) h$ is a splitting for $\nu$.  
\end{proof}

\begin{lemma}
\label{lemma-onto}
There is a $B$-$A$-bimodule isomorphism $Q \stackrel{\cong}{\longrightarrow} \Hom ({}_AP, {}_AA)$
given by $q \mapsto \nu( - \otimes_B q)$.
\end{lemma}
\begin{proof}
An inverse is given for $g \in \Hom ({}_AP, {}_AA)$ by $g \rightarrow \sum_i {q'}_i g({p'}_i) \in Q$ 
using eqs.~(\ref{eq1}) and~(\ref{eq2}). 

A second proof: from the hom-tensor adjoint relation,  the functor $\mathcal{P}$ has right adjoint
the functor (of coinduction) ${}_B\Hom ({}_AP, {}_A -)$ in general, and by hypothesis also the right adjoint $\mathcal{Q}$. From the  uniqueness of adjoint functors up to natural  isomorphism \cite[p.\ 247]{AF},  the lemma follows from $\mathcal{Q}(A) \cong \Hom ({}_AP, {}_AA)$. 
\end{proof}
\section{Preliminaries on separable equivalence}
We continue to establish fixed notation used in the rest of the paper.
\begin{definition}
Rings $A$ and $B$ are \textbf{separably equivalent} if $A$ separably divides $B$ with context split epi $\nu: P \otimes_B Q \rightarrow A$ and $B$ separably divides $A$ with context split epi $\mu: {}_BQ \otimes_A P_B \rightarrow {}_BB_B$.   Rings $A$ and $B$ are \textbf{symmetric separably equivalent} if $A$ symmetric separably divides $B$ with adjoint functors $(\mathcal{P} =  {}_AP \otimes_B - , \mathcal{Q} = {}_BQ \otimes_A - )$ and $B$ symmetric separably divides $A$ with adjoint functors
$(\mathcal{Q}, \mathcal{P})$ \cite[Def.\ 6.1, Remark 6.1]{LK1995}.  In other words,  the functors $\mathcal{P}$ and $\mathcal{Q}$ between $A$-Mod and $B$-Mod are Frobenius functors \cite{M67}, i.e., adjoint in either order. 
\end{definition}

Shared properties of separably equivalent rings include global dimension \cite[Corollary 6.1]{LK1995}
and other  homological properties of rings defined in \cite[Theorem 6.1]{LK1995}, such as a ring whose flat modules are projective, which characterises  von Neumann regular rings.

\begin{prop}
Suppose rings $A$ and $B$ are separably equivalent with context $(P,Q,\mu,\nu)$ defined above.  
There is a ring structure on $P \otimes_B Q$ with multiplication defined by 
$$
(p \otimes_B q)(p' \otimes_B q') = p \mu(q \otimes_A p') \otimes_B q',$$
and linearly extended (call it $\mu$-multiplication).  Then the ring $P \otimes_B Q$ is unital  if and only if $A$ and $B$ are symmetric separably equivalent.
\end{prop}
\begin{proof}
A short computation depending only on $\mu$ being an $B$-bimodule homomorphism shows that the $\mu$-multiplication is 
an associative multiplication.
($\Leftarrow$) By Proposition~\ref{prop-adjelts}  there is  $\sum_j p_j \otimes q_j \in P \otimes_A Q$ 
for the adjoint pair $(\mathcal{Q},\mathcal{P})$, which satisfy the identities  
\begin{eqnarray}
\label{eq3}
\sum_j \mu(- \otimes_A p_j)q_j & = & \id_Q \\
\label{eq4}
\sum_j p_j \mu(q_j \otimes_A -) & = & \id_P 
\end{eqnarray}
 Define $1_{P \otimes Q} = \sum_j p_j \otimes_B q_j$.  Then for all $p \in P, q \in Q$, $1_{P \otimes Q} (p \otimes q) = p \otimes q$ follows from eq.~(\ref{eq4}) and $(p \otimes q)1_{P \otimes Q} = p \otimes q$ follows
from eq.~(\ref{eq3}).  

($\Rightarrow$) Suppose there is an identity,  $1_{P\otimes Q} = \sum_j p_j \otimes_B q_j$.  Then letting
$$p' := \sum_j p_j\mu (q_j \otimes p) - p$$ for arbitrary $p \in P$, left unitality gives $p' \otimes q = 0$ for every $q \in Q$.
By Lemma~\ref{lemma-onto}, the homomorphism $Q \rightarrow \Hom ({}_AP, {}_AA)$ given by 
$q \mapsto \nu(- \otimes_B q)$ is onto.  Since ${}_AP$ is finite projective, there are dual bases
$x_k, f_k$ such that $p' = \sum_k f_k(p')x_k = \sum_k \nu(p' \otimes q_k) x_k = 0$ for some $q_k \in Q$.  Thus eq.~(\ref{eq4}) follows.  Eq.~(\ref{eq3}) is similarly derived from right unitality.  
\end{proof} 

Similarly, $Q \otimes_A P$ has $\nu$-multiplication and an identity element $1 = \sum_i {q'}_i \otimes_A {p'}_i$, the unit of adjunction of $(\mathcal{P}, \mathcal{Q})$ in Proposition~\ref{prop-adjelts} satisfying the equations~(\ref{eq1}) and~(\ref{eq2}).  

For an example consider a separable Frobenius extension $A \supseteq B$ with Frobenius homomorphism $E: A \rightarrow B$ where $E|_B = \id_B$ with dual bases $\{x_i \}, \{ y_i \}$ \cite[Prop.\ 6.1]{LK1995}.  Then $A$
and $B$ are symmetric separably equivalent with context $P = {}_AA_B$, $Q = {}_BA_A$, $\nu:
A \otimes_B A \rightarrow A$ multiplication and $E: A \rightarrow B$ both split epis.  The units of
adjunction in $P \otimes_B Q \cong A \otimes_B A$ and $Q \otimes_A P \cong A$ are the \textit{dual bases tensor} $\sum_i x_i \otimes_B y_i$ (i.e., satisfying $\sum_i E(ax_i)y_i = a = \sum_i x_i E(y_ia)$ for all $a \in A$) and $1$, respectively, with $A \otimes_B A$ having the $E$-multiplication.

\section{Structure for symmetric separably equivalent rings }

Recall that a \textit{Frobenius bimodule} is a bimodule ${}_AW_B$ such that
$W$ is left and right finite projective, and $\Hom (W_B, B_B) \cong \Hom ({}_AW, {}_AA)$ as natural $B$-$A$-bimodules \cite{CK, M67, NEFE}.  For example, a
Frobenius extension $A \supseteq B$ has the natural bimodule ${}_AA_B$, which
is a Frobenius bimodule;  a progenerator module $M_T$ gives rise to the natural (Morita context) bimodule
${}_EM_T$, where $E = \End M_T$, which is a Frobenius bimodule (\cite[Theorem 1.1]{M67}, \cite[Theorem 2.3]{NEFE} and \cite[p.\ 261]{AF}).   

Given rings $A$ and $B$ linked by bimodule ${}_AP_B$, recall that \textit{$A$ is ${}_AP_B$-separable over $B$} if the evaluation mapping $e$, 
\begin{equation}
\label{eq: sep}
{}_AP \otimes_B \Hom ({}_AP, {}_AA)_A   \stackrel{e}{\rightarrow} {}_AA_A
\end{equation}
is a bimodule split epi \cite{KS}.  For example, a ring extension $A \supseteq B$
 is by definition \textit{separable} if the natural bimodule ${}_AA_B$ satisfies the split
epi condition in Eq.~(\ref{eq: sep}); or \textit{split} if the natural bimodule ${}_BA_A$
satisfies the split epi condition above \cite{NEFE}.  

\begin{theorem}
\label{theorem-list}
Suppose rings $A$ and $B$ are symmetric separably equivalent with context
$$({}_AP_B, {}_BQ_A,\nu: P \otimes_B Q \rightarrow A, \mu: Q \otimes_A P \rightarrow B, \sum_j {q'}_i \otimes_A {p'}_i, \sum_j p_j \otimes_B q_j),$$
satisfying Eqs.~(\ref{eq1}),~(\ref{eq2}),~(\ref{eq3}) and~(\ref{eq4}).  
Then 
\begin{enumerate}
\item ${}_AP_B$ is a Frobenius bimodule;
\item $A$ is $P$-separable over $B$;
\item there is a ring isomorphism  $\End P_B \cong P \otimes_B Q$ ;
\item the  mapping $A \stackrel{\lambda}{\into} \End P_B$ is a split, separable Frobenius extension. 
\end{enumerate} 
Conversely, (4) with the hypothesis that $P_B$ is a progenerator and ${}_AP$
finite projective implies that $A$ and $B$ are symmetric separably equivalent. 
\end{theorem}
\begin{proof}
The bimodule ${}_AP_B$ satisfies by Lemma~\ref{lemma-onto}
\begin{equation}
\label{eq: cue}
{}_B\Hom (P_B,B_B)_A \cong {}_BQ_A \cong {}_B\Hom ({}_AP, {}_AA)_A
\end{equation}
via $q \mapsto \mu(q \otimes_A -)$ and the mapping $q \mapsto \nu(-\otimes_B q)$ in the lemma. Since $P$ is left and right finite projective (indeed progenerator), it 
is a Frobenius bimodule.

The evaluation mapping $e$ in Eq.~(\ref{eq: sep}) forms a commutative triangle  with
the context split epi $\nu: P \otimes_B Q \rightarrow A$ via the bimodule
isomorphism in Eq.~(\ref{eq: cue}).  Thus, $A$ is $P$-separable over $B$.

The mapping $P \otimes_B Q \rightarrow \End P_B$ given by 
$p \otimes_B q \mapsto p\mu(q \otimes_A -)$ is an isomorphism of $A$-bimodules, since $P_B$ is finite projective, so $\End P_B \cong P \otimes_B \Hom (P_B, B_B)$ as $A$-$A$-bimodules.  But $Q \cong \Hom (P_B, B_B)$
via $q \mapsto \mu(q \otimes_A -)$ as noted before.
The inverse mapping, $\End P_B \rightarrow P \otimes_B Q$ is in fact 
given by $ \alpha \mapsto \sum_j \alpha(p_j) \otimes_B q_j$. (For example, $ \sum_j \alpha(p_j) \mu(q_j \otimes_A -) = \alpha$ by eq.~(\ref{eq4}).) This is a ring isomorphism since
\begin{eqnarray*}
 (\sum_k \beta(p_k) \otimes_B q_k)(\sum_j \alpha(p_j) \otimes_B q_j) 
& = & \\
\sum_{k,j} \beta(p_k\mu(q_k \otimes_A \alpha(p_j))) \otimes_B q_j 
&= &   \sum_j \beta(\alpha(p_j)) \otimes_B q_j 
\end{eqnarray*}
for $\alpha, \beta \in \End P_B$.  

 Since ${}_AP_B$ is a Frobenius bimodule, it follows from the endomorphism ring theorem in \cite[Theorem 2.5]{NEFE} that
 $\lambda: A \into \End P_B$ is a Frobenius extension of
rings.  Under an identification of
$\End P_B$ and $P \otimes_B Q$ in the previous paragraph, $A$ embeds via
$a \mapsto \sum_j ap_j \otimes_B q_j$, and $\nu: P \otimes_B Q \rightarrow A$
is a Frobenius homomorphism, with dual bases tensor
\begin{equation}
\label{dualbasestensor}
\sum_{i,j} p_j \otimes_B {q'}_i \otimes_A {p'}_i \otimes_B q_j 
\end{equation}
by an amusing exercise.  It follows from theorems in \cite{KS} that
$A \into \End P_B$ is a split extension, and the anti-monomorphism $\rho: A \into \End {}_BQ$
is a separable extension (once you apply what's proven for $P$ to $Q$, see also \cite[Theorem 3.1]{LK99}).  But there is an anti-isomorphism of rings,  $\End {}_BQ \cong \Hom ({}_BQ, {}_BB) \otimes_B Q \cong P \otimes_B Q$ (using Lemma~\ref{lemma-onto}, $p \mapsto \mu(- \otimes_A p)$) given
by $\beta \in \End {}_BQ$ is mapped into $\sum_j p_j \otimes_B \beta(q_j)$.  
From the identity~(\ref{eq: centered}), 
it follows that the ring extensions  $A \into \End P_B$ and $A \into \End {}_BQ$
are isomorphic in the usual terms of commutative diagrams. Hence $A \into \End P_B$ is also a separable
ring extension.

The converse (4) $\Rightarrow$ symmetric separable equivalence of $A$ and $B$ follows from 
\cite[Theorem 2.8]{NEFE}, which shows that ${}_AP_B$ is a Frobenius bimodule, and by theorems in \cite[Theorem 1, Proposition 2]{KS} (corrected and simplified in \cite[Theorem 3.1]{LK99}), which show that ${}_AP_B$ is also biseparable.  The result follows from Theorem~\ref{theorem-hollywood}
(2 $\Rightarrow$ 1) below. Note that
\cite[Theorem 3.1(4)]{LK99} is applied at one point using $P' := {}_{B^{\rm op}}P_{A^{\rm op}}$ and that $B^{\rm op}$ is $P'$-separable over $A^{\rm op}$ unwinds as 
$B$ is $\Hom (P_B, B_B)$-separable over $A$.  
\end{proof} 
Similarly (or symmetrically), ${}_BQ_A$ is a Frobenius bimodule, $B$ is $Q$-separable over $A$,
$Q \otimes_A P \cong \End Q_A$ with  $\nu$-multiplication and $B \into \End Q_A$ is a split, separable Frobenius extension.  

The theorem is useful for showing which  properties of symmetrical separably equivalent rings are shared.  Since $B$ and $\End P_B$ are Morita equivalent, and $A \into \End P_B$ has very special properties, it follows that properties like ``polynomial identity algebra'' that are Morita invariants and
descend along algebra monomorphisms, are then shared properties of this weaker equivalence.  Also Morita invariants like global dimension that descend along projective split, separable or Frobenius
extensions are symmetric separable invariants:   \cite[Theorem 5.1]{LK1995} shows
 that $D(A) =  D(B)$ from the Theorem, part (4) above.
See also \cite[Theorems 5.2 and 6.1]{LK1995} and \cite[Theorem 2.6]{CK}.  

\section{Characterising symmetric separable equivalence}

In this section we characterise symmetric separable equivalence. One more definition is useful for this objective. 

\begin{definition}
A left and right finite projective bimodule ${}_AV_B$ is called \textbf{biseparable} if $A$ is $V$-separable over $B$ and $B$ is $\Hom (V_B, B_B)$-separable over $A$ \cite[Definition 2.4]{CK}. 
Using a reflexivity of finite projective modules, this is equivalent to the two evaluation mappings below  being split bimodule epimorphisms, 
\begin{enumerate}
\item ${}_AV \otimes_B \Hom ({}_AV, {}_AA)_A \stackrel{e_1}{\longrightarrow} {}_AA_A$ ;
\item ${}_B\Hom (V_B, B_B) \otimes_A V_B \stackrel{e_2}{\longrightarrow} {}_BB_B$. 
\end{enumerate}
\end{definition}

An example of the definition is the natural bimodule ${}_AA_B$ for
a split, separable, right and left finite projective ring extension $A \supseteq B$ \cite{CK}.

\begin{theorem}
\label{theorem-hollywood}
The following are equivalent conditions on two rings $A$ and $B$:
\begin{enumerate}
\item $A$ and $B$ are symmetric separably equivalent;
\item $A$ and $B$ are linked by a Frobenius biseparable bimodule ${}_AP_B$;
\item $A$ and $B$ are separably equivalent with at least one context bimodule Frobenius.
\end{enumerate}
\end{theorem}
\begin{proof}
(1 $\Rightarrow$ 2) Theorem~\ref{theorem-list} informs us that a context bimodule ${}_AP_B$ is Frobenius and $A$ is $P$-separable over $B$.  Also $B$ is $Q$-separable over $A$.  Since context bimodule ${}_BQ_A \cong {}_B\Hom (P_B, B_B)_A$ by Lemma~\ref{lemma-onto}, and context bimodule mapping
$\mu$ and $\nu$ are split epis, it follows that ${}_AP_B$ is biseparable. 

(2 $\Rightarrow$ 1) Suppose ${}_BQ_A$ is either the left or right dual of ${}_AP_B$, unique up to isomorphism. Then the associated functors between categories of modules  formed by tensoring and denoted as before by $\mathcal{P}$ and $\mathcal{Q}$ are Frobenius functors, adjoint functors in either order. 
The counits of adjunction are split epis by Lemma~\ref{lemma-oneforall}, since there are split epis $P \otimes_B Q \rightarrow A$ and $Q \otimes_A P \rightarrow B$ from the hypothesis that ${}_AP_B$ is biseparable.  

(3 $\Rightarrow$ 2) Suppose rings $A$ and $B$ are separably equivalent with context $({}_AP_B, {}_BQ_A,\,$ $ \nu\!:\! P \otimes_B Q \rightarrow A,\, \mu\!:\! Q \otimes_A P \rightarrow B)$. We show that bimodule $P$ is biseparable.  Define a mapping $h: Q \rightarrow \Hom ({}_AP, {}_AA)$ by
$q \mapsto \nu(- \otimes_B q)$. Note that $h$ is a $B$-$A$-bimodule homomorphism. Note that the evaluation mapping $e_1: P \otimes_B \Hom ({}_AP,{}_AA) \rightarrow A$ satisfies $\nu = e_1 \circ (\id_P \otimes_B h)$.  It follows that $e_1$ is a split
epimorphism of $A$-$A$-bimodules, since $\nu$ is. 
Similarly, define a mapping $g: Q \rightarrow \Hom (P_B, B_B)$ by $q \mapsto \mu(q \otimes_A -)$,
a $B$-$A$-bimodule homomorphism. The evaluation mapping $e_2: \Hom (P_B, B_B) \otimes_A P \rightarrow B$ satisfies $\mu = e_2 \circ (g \otimes_B \id_P)$, whence $e_2$ is a split epimorphism
of $B$-$B$-bimodules. \end{proof}
Note that the proof of 3 $\Rightarrow$ 2 establishes the following.

\begin{prop} Separably equivalent rings 
are linked by a biseparable bimodule.
\end{prop}

  When a biseparable bimodule linking
two rings is also Frobenius, is the question in  \cite[Problem 2.8]{CK}.  The lemma below provides an affirmative answer when both rings are symmetric algebras.  

A symmetric algebra is an algebra with nondegenerate trace, or Frobenius algebra with Nakayama automorphism an inner automorphism: for example, a semisimple algebra, a group algebra or  a unimodular Hopf algebra with square antipode that is conjugation by a grouplike element \cite{NEFE}.    The two notions of separable equivalence coincide on the class of symmetric algebras \cite{P}. 
Suppose $k$ is the ground field of $A$ and $B$.  Then a bimodule ${}_AP_B$ has
a third $B$-$A$-bimodule dual, $P^* := \Hom (P, k)$. In these terms, an algebra $A$ is symmetric if  ${}_AA_A \cong {}_A{A^*}_A$.

In general, a Frobenius algebra $A$, with  Nakayama automorphism $\alpha: A \rightarrow A$, satisfies a bimodule isomorphism,  ${}_AA_A \cong {}_{\alpha}{A^*}_A$, where an $A$-module ${}_AM$ receives a twist
${}_{\alpha}M$ by defining $a . m = \alpha(a)m$ for every $a\in A, m \in M$. 
If $\alpha$ is an inner automorphism, it is an easy exercise to show that
${}_{\alpha}M \cong {}_AM$.  Recall that where $\alpha$ and $\beta$ are automorphisms of $A$ and $B$, an $\alpha$-$\beta$-\textit{Frobenius bimodule} is a bimodule ${}_AW_B$ such that
$W$ is left and right finite projective, and there is a $B$-$A$-bimodule isomorphism ${}_B\Hom (W_B, B_B)_A \cong {}_{\beta}\Hom ({}_AW, {}_AA)_{\alpha}$  \cite{CK, M67, LK99}.

\begin{lemma}
\label{lemma-FrobeniusBimodule}
Suppose $A$ and $B$ are Frobenius algebras with Nakayama automorphisms $\alpha$
and $\beta$, respectively.  Then any  left and right finite projective bimodule ${}_AP_B$ is an
$\alpha$-$\beta$-Frobenius bimodule.  In particular, if $A$ and $B$ are symmetric algebras, $P$ is
automatically a Frobenius bimodule. 
\end{lemma}
\begin{proof}
Suppose $B$ and $A$ are symmetric algebras. Then the three duals of a bimodule ${}_AP_B$ coincide, since 
$${}_B\Hom (P_B, B_B)_A \cong {}_B\Hom (P_B, {B^*}_B)_A \cong {}_B{P^*}_A$$
and similarly 
${}_B\Hom ({}_AP, {}_AA)_A \cong {}_B{P^*}_A$ (making use of the hom-tensor adjoint relation \cite[20.6]{AF}).  The general argument is similar with more careful bookkeeping using the calculus of twisted modules in for example  \cite{LK99}, and more careful use of \cite[20.6]{AF}.
\end{proof}
The lemma implies that a subalgebra pair of Frobenius algebras $A \supseteq B$ with Nakayama automorphisms $(\alpha,\beta)$, where $A_B$ is finite projective, forms an $(\alpha,\beta)$-Frobenius extension as  defined on \cite[p.\ 401]{LK99}; e.g., a  left coideal subalgebra of a finite-dimensional Hopf algebra. 
The following is  obvious from  Lemma~\ref{lemma-FrobeniusBimodule} and Theorem~\ref{theorem-hollywood}. 
\begin{cor}
A separable equivalence between symmetric algebras $A$ and $B$ is a symmetric separable equivalence.
\end{cor}

The construction for any algebra $A$ as a subalgebra of a symmetric algebra structure on 
$A \times A^*$ is also noteworthy in this context \cite[p.\ 443]{Lam}.
It is interesting to extend the theory in this paper from module categories and exact functors such as $\mathcal{P}$ and $\mathcal{Q}$ to general abelian, additive or exact categories (cf.\ \cite{LK1995, P}).  


\begin{thebibliography}{XXXXXX}
\begin{small}
\bibitem{AF} F.W.~Anderson and K.R.~Fuller,  \textit{Rings and Categories of Modules}, G.T.M.\ \textbf{13}, Springer, 1992.  
\bibitem{BE} P.E.~Bergh and K.~Erdmann, The representation dimension of Hecke algebras and symmetric groups, \textit{Advances in Math.} \textbf{228} (2011), 2503--2521.  
\bibitem{CK}S.~Caenepeel and L.~Kadison, Are biseparable extensions Frobenius? \textit{K-Theory} \textbf{24} (2001), 361--383. 
\bibitem{IS}M.C.~Iovanov and A.~Sistko, Maximal subalgebras of finite-dimensional algebras, arXiv preprint \texttt{1705.00762}. 
\bibitem{K91}L.~Kadison, On global dimension, tower of algebras and Jones index, U.C.\ Berkeley and Roskilde University Centre preprint, 1991,  \texttt{http://milne.ruc.dk/imfufatekster/pdf/210.pdf}
\bibitem{K93}L.~Kadison, Algebraic aspects of the Jones basic construction,
	 \textit{Comptes Rendus Math.\ Reports Acad.\ Sci.\ Canada}  
\textbf{15} (1993), 223 - 228.
\bibitem{LK1995}L.~Kadison, On split, separable subalgebras with counitality condition, \textit{Hokkaido Math.\ J.} \textbf{24} (1995), 527--549. 
\bibitem{NEFE}L.~Kadison, \textit{New Examples of Frobenius Extensions}, University Lecture Series \textbf{14}, A.M.S., 1999.
\bibitem{LK99}L.~Kadison, Separability and the twisted Frobenius bimodule,
\textit{Alg.\ Rep.\ Theory} \textbf{2} (1999), 397--414. 
\bibitem{Lam}T.Y.~Lam, \textit{Lectures on Modules and Rings}, Grad.\ Texts Math.\ \textbf{189}, Springer Verlag, 1999. 
\bibitem{L}M.~Linckelmann, Finite generation of Hochschild cohomology of Hecke algebras of finite classical type in characteristic zero, \textit{Bull.\ Lond.\ Math.\ Soc.} \textbf{43} (2011), 871--885. 
\bibitem{M}S.~Mac Lane, \textit{Categories for the Working Mathematician}, Springer GTM \textbf{5}, 2nd edition, 1997.
\bibitem{M67}K.~Morita,
The endomorphism ring theorem for Frobenius extensions,
\textit{Math.\ Zeit.} \textbf{102 } (1967), 385--404.
\bibitem{P}S.~Peacock, Separable equivalence, complexity and representation type, \textit{J.\ Algebra} \textbf{490} (2017), 219--240.
\bibitem{R}M.D.~Rafael, Separable functors revisited, \textit{Comm.\ Alg.} \textbf{18} (1990), 1445--1459. 
\bibitem{KS}K.~Sugano, Note on separability of endomorphism rings, \textit{J.\ Fac.\ Sci.\ Hokkaido Univ.} \textbf{21} (1971), 196--208. 
\end{small}
\end{thebibliography}
\end{document}